\documentclass[12pt]{amsart}
\usepackage{amssymb,amsmath,amsthm,graphicx}
\newtheorem{theorem}{Theorem}

\newtheorem*{lemma*}{Lemma}
\newtheorem{lemma}{Lemma}
\newtheorem*{theorem*}{Theorem}
\newtheorem{proposition}{Proposition}

\newcommand{\tF}{\tilde{F}}

\newcommand{\ds}{\displaystyle}

\newcommand{\ep}{\varepsilon}

\newcommand{\cH}{\mathcal{H}}

\newcommand{\K}{{\bf K}}

\newcommand{\ka}{\kappa}
\newcommand{\sn}{\hbox{sl}}
\newcommand{\cn}{\hbox{cl}}

\begin{document}
\title{Bounding Solutions of a Forced Oscillator}
\author{Kenneth R. Meyer}
\thanks{ Research partially supported by the Charles Phelps Taft Foundation.}
\address{Department of Mathematical Sciences \\
University of Cincinnati \\
Cincinnati, Ohio 45221-0025}
\email{ken.meyer@uc.edu}
\author{Dieter S Schmidt}
\address{Department of   Computer Science \\
University of Cincinnati \\
Cincinnati, Ohio 45221-0030}
\email{dieter.schmidt@uc.edu}

\begin{abstract}
We show that all solutions are bounded for a periodically forced nonlinear oscillator by using a special set of coordinates, simplifying the system with a convergent Lie transformation and then by showing the period map has large invariant curves  by Moser's invariant curve theorem.
\end{abstract}

\date{\today}

\maketitle

\centerline{To the memory of George Sell, our colleague and friend.}

\section*{Introduction}
In 1976 G. R. Morris \cite{granger} showed that all solutions of
$$
\ddot{x}+2x^3=p(t),
$$
are bounded when $p(t)$ is continuous and periodic. This gave rise to an abundance of generalizations  \cite{eddie,mark1,mark2,liu1,liu2,yuan}  with references to many more.
All these proofs depend on showing that Moser's invariant curve theorem \cite{jurgen} implies the existence of invariant curves near infinity for the period map.  We will too.

Here we  give a simple direct proof of the generalization of Morris' result found in Dieckerhoff and  Zehnder \cite{eddie}, but our proof  requires much less differentiability.  
\vspace{0.1cm}
\noindent
\begin{theorem}\label{mainthm}
For any integer $n>1$ all solutions of
\begin{equation}\label{maineq}
\ddot{x} + nx^{2n-1}= p_0(t)+p_1(t)x+\cdots+p_{2n-2}(t)x^{2n-2}
\end{equation}
are bounded where the coefficients  of the polynomial on the right hand side are T-periodic and satisfy $p_0(t)\in C^0$, $p_j(t)\in C^1$ for $j=1,\ldots,n-1$ and $p_j(t)\in C^2$ for $j=n,\ldots,2n-2$.
\end{theorem}
 
Our proof is  based on a  generalization of action-angle variables,  a convergent Lie transformation, and Moser's invariant curve theorem.   
As an overall outline we give a quick proof of Morris' original theorem. Then the  full theorem is established in two steps.  
In the  first step we define a special case and then follow the proof given for Morris' Theorem to prove boundedness for the special  case.  In the second step we use the method of Lie transforms to reduce the general case to the special case.

\section*{Action--Angle Variables}

Let $\sn (\ka)$ be the solution of the reference equation 
$$
\xi '' +n\xi^{2n-1}=0,\,\,\,\xi(0)=0,\,\,\,\xi'(0)=1,
$$
where $\ds  '=\frac d{d\ka}$ and let $\cn (\ka)=\sn'(\ka)$.  When $n=1$ these are  the standard  sine and cosine functions and when $n=2$ these are the  lemniscate functions \cite{ww}.

The Hamiltonian for this equation is
$$
L=\frac 12 \eta^2 +\frac 12 \xi^{2n},
$$
where $\eta=\xi'$. Since the level sets of $L$ are ovals these solutions are periodic.  

Both $\sn (\ka )$ and $-\sn(-\ka )$ satisfy  the reference equation and the same  initial conditions,   so  by  the
uniqueness    theorem    for    ordinary    differential     equations,
$\sn (\ka ) = -\sn(-\ka)$, i.e., $\sn(\ka )$ is odd.
As $\ka$ increases from zero, $\sn (\ka )$ increases  from  zero  until  it
reaches its maximum  value  of  1  after  some  time $\tau>0$.  Take the equation $L=1/2$, solve for $\eta = d\xi/d\ka$, separate variables, and integrate to get
$$
\tau = \int_0^1 \frac {d\xi}{\sqrt{1-\xi^{2n}}}.
$$
Both $\sn (\tau+\ka )$ and $\sn (\tau-\ka)$ satisfy the equation and the same initial condition when $\ka=0$,  so by uniqueness of the solutions  of  differential  equations  it
follows that $\sn (\tau +\ka) = \sn (\tau-\ka)$, or that $\sn(\ka)$ is even about $\tau$, $\sn (\ka )$ is $4\tau $ periodic and  odd
harmonic.  
It is clear that $\sn(\ka )$  is increasing for $-\tau  < \ka < \tau $ and that 
$\sn''(\ka ) > 0$ (so $\sn (\ka )$ is convex) for $-\tau  < \ka < 0,$ and 
${\sn}''(\ka )< 0$ (so $\sn(\ka )$ is concave) for $0 < \ka < \tau $.  Thus, $\sn(\ka )$ has the same basic
symmetry properties as the sine function  with $4\tau$ playing the role of $2\pi$. In a like manner $\cn(\ka)$ has the same basic symmetry properties as the cosine function.

In summary
$$
\sn(0) =0,\,\,\, \sn' ( 0)=1,\,\,\,\cn (0)=1,\,\,\,\cn' (0)=0;
$$
$$
\cn^2(\ka)+\sn^{2n}(\ka)=1;
$$
$$
\sn'(\ka)=\cn(\ka),\quad \cn'(\ka)=-n\; \sn^{2n-1}(\ka);
$$
$$
\sn (\ka ) = -\sn(-\ka), \quad \sn (\tau+\ka ) = \sn (\tau-\ka), \quad \sn(\ka)=\sn(\ka+4\tau);
$$
$$
\cn (\ka ) = \cn(-\ka), \quad \cn (\tau+\ka) = -\cn (\tau-\ka), \quad \cn(\ka)=\cn(\ka+4\tau).
$$

To get {\em action--angle} variables $(K,\ka)$,  let
$$
x= K^{\frac 1{n+1}}\sn(\ka), \qquad y=-K^{\frac n{n+1}}\cn(\ka),
$$
and check 
$$
\begin{array}{rcl}
dx\wedge dy &= & \ds
\left( \frac 1{n+1}K^{-\frac {n}{n+1}}\sn(\ka)dK+K^{\frac 1{n+1}}\cn(\ka)d\ka \right) \wedge
\\ && \qquad \qquad \ds \left(\frac {-n}{n+1}K^{-\frac {1}{n+1}}\cn (\ka)dK  +nK^{\frac n{n+1}} \sn^{2n-1}(\ka)d\ka \right),
\\ \\ &=&  \ds
\frac n{n+1}\sn^{2n}(\ka) dK\wedge d\ka - \frac n{n+1}\cn^2(\ka)d\ka \wedge  dK,
\\ \\ &=& \ds
\frac n{n+1}dK\wedge d\ka,
\end{array}
$$
which  is symplectic with multiplier $\ds \frac {n+1}n$.  

In these coordinates 
$$
L  =\frac 12 (y^2+x^{2n}) =  \frac {n+1}{n}\;\frac 12 \left( K^{\frac {2n}{n+1}}\cn^2(\ka) +K^{\frac {2n}{n+1}}\sn^{2n}(\ka)\right) = 
\frac {n+1}{2n} K^{\frac {2n}{n+1}}.
$$

The Hamiltonian for equation (\ref{maineq}) is
\begin{equation}\label{originalham}
H=\frac 12 (y^2+  x^{2n})-\sum_{j=0}^{2n-2}p_j(t)\frac {x^{j+1}}{(j+1)}
\end{equation}
and in action-angle variables,
\begin{equation}\label{originalKk}
H=\frac {n+1}{2n} K^{\frac{2n}{n+1}}+\sum_{j=1}^{2n-1}K^{\frac{j}{n+1}}f_j(\ka,t) 
\end{equation}
where  we have set 
$ \ds f_j(\ka,t)= -\frac{n+1}{j\;n}\;\sn^j(\ka)\;p_{j-1}(t) $

\section*{Proof of Morris' Theorem}

In order to illustrate  how Moser's invariant curve theorem can be applied with the action-angle coordinates $(K,\ka)$ in a simple setting
let us prove  Morris' original theorem. So look at the  equation,
$$
\ddot{x}+2x^3=p(t),
$$
were $p(t)$ is continuous and T-periodic.
In action-angle variables,  $(K,\ka)$, the Hamiltonian is
$$
H=\frac 34 K^{4/3}-\frac 32K^{1/3} \sn (\ka)p(t),
$$
and the equations of motion are
$$
\begin{array}{l}
\ds
\dot{K}=-\frac 32K^{1/3} \cn (\ka)p(t),
\\ \\ \ds
\dot{\ka}=-K^{1/3}+\frac 12 K^{-2/3}\sn(\ka)p(t).
\end{array}
$$

Let $\ds \Lambda=K^{1/3}$ so the equations become 

$$
\begin{array}{l}
\ds
\dot{\Lambda}=-\frac 12\Lambda^{-1} \cn (\ka)p(t),
\\ \\ \ds
\dot{\ka}=-\Lambda+\frac 12 \Lambda^{-2}\sn(\ka)p(t).
\end{array}
$$
First note that since $\sn,\,\cn,$ and $p$ are all uniformly bounded these equations are analytic for all $\ka,\,t$ and $\Lambda >0$.

Integrate from 0 to $-T$ to compute the period map $\mathcal{P}:(\Lambda,\ka)\to (\Lambda^*,\ka^*)$ were 
$$
\begin{array}{l}
\Lambda^*=\Lambda + F(\Lambda,\ka),
\\ \\
\ka^*=\ka +T\Lambda+G(\Lambda,\ka),
\end{array}
$$where $F(\Lambda,\ka)=O(\Lambda^{-1}),\,\,G(\Lambda,\ka)=O(\Lambda^{-2})$.

An {\em encircling curve} is a curve, $\mathcal{C}$,  of the form  $K=\phi(\ka)$ (or  $\Lambda=\phi(\ka)$) were $\phi$ is continuous, $4\tau$-periodic, positive and near a circle about the origin.
The invariant curve  theorem  requires the the image of an encircling curve must intersect itself. 
The coordinates $(K,\ka)$ are symplectic and so the period map is area preserving and thus the image under $\mathcal{P}$ of any encircling curve must intersect itself in the $(K,\ka) $
coordinates.  The map from $K$ to $\Lambda$ is invertible from $K>0$ to $\Lambda>0$ so the image of an encircling curve in $(\Lambda,\ka)$ coordinates must intersect itself.   The curve is invariant if $\mathcal{P}(\mathcal{C})=\mathcal{C}.$

The invariant curve theorem requires that the functions $F,\, G$ have at least $\ell$ derivatives were originally $\ell=333$, but recent work has reduced it to $\ell=5$.
No matter our system is analytic and we only need a continuous invariant curve.

 Let   $\mathcal{A}(a)$ be the annulus $\{ (\Lambda,\ka): 1 < a\le \Lambda  \le a+1\}$ and define the $\ell^{th}$ norm of a function $R$ on $\mathcal{A}(a)$ to be
$$
|R|_\ell = \sup \left| \left( \frac {\partial}{\partial \Lambda}\right)^{\sigma_1} \left( \frac {\partial}{\partial \ka}\right)^{\sigma_2}R(\Lambda,\ka)\right|
$$
where the $\sup$ is over all $0 \le \sigma_1+\sigma_2 \le \ell$ , and all $(\Lambda,\ka) \in \mathcal{A}(a)$.  
From the form of the equations
$$
|F|_\ell = O(a^{-1}), \,\,\,\,\,\,\,|G|_\ell = O(a^{-2}).
$$

Moser's theorem says there is a $\delta >0$  depending on the given data such that if $|F|_\ell < \delta,\,\, |G|_\ell <\delta$ then there is an invariant encircling curve in 
$\mathcal{A}(a)$.   From the above there is an $a^*$ such that for all $a>a^*$ we have $|F|_\ell < \delta,\,\, |G|_\ell <\delta$ on $\mathcal{A}(a)$.  So the period map
$\mathcal{P}$ has arbitrarily large invariant encircling curves, all solutions are bounded and Morris' Theorem is established.

\section*{The Special Case}

Our special case is  the Hamiltonian in action-angle variables of the form
\begin{equation}\label{HSCKk}
\cH=\frac {n+1}{2n} K^{\frac{2n}{n+1}}+\sum_{j=2}^{2n-1}K^{\frac{j}{n+1}}\overline{f}_j(t)+\sum_{j=-\infty}^{1}K^{\frac{j}{n+1}}f_j(\ka,t),
\end{equation}
where 
\begin{enumerate}
\item 
$\overline{f}_j(t)$ is continuous and $T$-periodic in $t$ for $j=2,\ldots,2n-1$, \vspace{5pt} 
\item 
$f_j(\ka,t)$  is analytic and $4\tau$-periodic in $\ka$, continuous and $T$-periodic in $t$  for $j=-\infty,\ldots, 1$, \vspace{5pt}
\item
 the infinite series in (\ref{HSCKk}) is uniformly convergent for  $K\ge \K$ and all $\ka,\,t$  with $\K>0$  a constant. \vspace{5pt}
\end{enumerate}

In $\cH$ the $\ka$ dependence has been removed from some terms to  facilitate  the proof at a cost of many extra terms. 
The extra terms are created when the original  Hamiltonian (\ref{originalKk}) is normalized in the next section.

\begin{proposition}
All solutions of the equations with Hamiltonian (\ref{HSCKk}) are bounded.
\end{proposition}

\begin{proof}

The equations of motion are
$$ 
\begin{array}{rcl}
\ds
\dot{K}&=& \ds  \sum_{j=-\infty}^{1}K^{\frac{j}{n+1}}\frac {\partial f_j(\ka,t)}{\partial \ka},
\\ \\ \ds
\dot{\ka}&= & \ds -K^{\frac {n-1}{n+1}}  
-\sum_{j=2}^{2n-1}\left(\frac {j}{n+1}\right)K^{\frac{j-n-1}{n+1}}\overline{f}_j(t)
-\sum_{j=-\infty}^{1}\left(\frac {j}{n+1}\right)K^{\frac{j-n-1}{n+1}}f_j(\ka,t).
\end{array}
$$
Set $\Lambda=K^{\frac{n-1}{n+1}}$ so that the differential equations are now
$$ 
\begin{array}{rcl}
\ds
\dot{\Lambda}&=& \ds \frac{n-1}{n+1} \sum_{j=-\infty}^{1}\Lambda^{\frac{j-2}{n-1}}\frac {\partial f_j(\ka,t)}{\partial \ka},
\\ \\ \ds
\dot{\ka}&= & \ds -\Lambda  
-\sum_{j=2}^{2n-1}\left(\frac {j}{n+1}\right)\Lambda^{\frac{j-n-1}{n-1}}\overline{f}_j(t)
-\sum_{j=-\infty}^{1}\left(\frac {j}{n+1}\right)\Lambda^{\frac{j-n-1}{n-1}}f_j(\ka,t).
\end{array}
$$
The above are series in $\Lambda^{\frac 1{n-1}}$ which are convergent for large $\Lambda$.  The two infinite series will be
treated as perturbations and the finite series in $\dot{\ka}$ contributes  to the twist term.
 The dominate term in the infinite series for $\dot{\Lambda}$ is  of order $\Lambda^{\frac {-1}{n-1}}$ and for
$\dot{\ka}$ it is  of order $\Lambda^{\frac {-n}{n-1}}$.  

We are interested in large $K$, that is large  $\Lambda$ so that 
$$ 
\begin{array}{rcl}
\ds
\dot{\Lambda}&=& O(\Lambda^{\frac {-1}{n-1}}),
\\ \\ \ds
\dot{\ka}&= & \ds -\Lambda -\sum_{j=2}^{2n-1}\left(\frac {j}{n+1}\right)\Lambda^{\frac{j-n-1}{n-1}}\overline{f}_j(t)+O(\Lambda^{\frac {-n}{n-1}}),
\end{array}
$$
where the estimates are on  $K\ge \K$.
Integrating from $0$ to $-T$ to compute the period map $\mathcal{P}:(\Lambda,\ka)\to (\Lambda^*,\ka^*)$ gives
$$
\begin{array}{l}
\Lambda^*=\Lambda+F(\Lambda,\ka),
\\ \\ \ds
\ka^*=\ka +\alpha(\Lambda) +G(\Lambda,\ka),
\end{array}
$$
where $F(\Lambda,\ka)=O(\Lambda^{\frac {-1}{n-1}}),\,\, G(\Lambda,\ka)=O(\Lambda^{\frac {-n}{n-1}})$. The twist term is 
$$
\alpha(\Lambda)=T\Lambda +\sum_{j=2}^{2n-1} \sigma_j \Lambda^{\frac{j-n-1}{n-1}},\,\,\hbox{ with }\,\,
\sigma_j=\frac j{n+1} \int_0^{-T} \overline{f}_j(t)dt,
$$ 
and its derivative is 
$$
\frac {d\alpha(\Lambda)}{d\Lambda}= 
T +\sum_{j=2}^{2n-1}\frac{j-n-1}{n-1}\sigma_j\Lambda^{\frac {j-2n}{n-1}}=T +O(\Lambda^{\frac {-1}{n-1}}).
$$

Consider the annulus $\mathcal{A}(a)=\{(\Lambda,\ka): 1 < a \le \Lambda \le a+1\}$.   
There is an $a^*>1$ such that
$\frac 12 T < d\alpha(\Lambda)/d\Lambda <2T$ on $\mathcal{A}(a)$ when $a>a^*$.
Moser's theorem says there is a $\delta >0$  depending on the given data such that if $|F|_\ell < \delta,\,\, |G|_\ell <\delta$ then there is an invariant encircling curve in 
$\mathcal{A}(a)$.   From the above there is an $a^{**}>a^*$ such that for all $a>a^{**}$ we have $|F|_\ell < \delta,\,\, |G|_\ell <\delta$ on $\mathcal{A}(a)$.  So the period map
$\mathcal{P}$ has arbitrarily large invariant encircling curves, all solutions are bounded and the Proposition is established.
\end{proof}

\section*{Reduction to the Special Case} 
In this final section we will finish the proof of the Theorem \ref{mainthm} by proving

\begin{proposition}
There exists an invertible  symplectic change of variables which transforms the original Hamiltonian $H$ in (\ref{originalKk}) 
to the special Hamiltonian $\cH$ in (\ref{HSCKk}).
\end{proposition}

More precisely we will construct two symplectic changes of variables, where the first one transforms the original Hamiltonian into an intermediate Hamiltonian of the form
\begin{equation} \label{intermediateKk}
\cH_1=\frac {n+1}{2n} K^{\frac{2n}{n+1}}+\sum_{j=n+1}^{2n-1}K^{\frac{j}{n+1}}\overline{f}_j(t)+\sum_{j=-\infty}^{n}K^{\frac{j}{n+1}}f_j(\ka,t).
\end{equation}
The second symplectic transformation will then transform (\ref{intermediateKk}) into the special Hamiltonian $\cH$ given in (\ref{HSCKk}).  Note that  $\overline{f}_j(t)$ and ${f}_j(\ka,t)$  represent  different functions from those in $\cH$, but with the same properties  as described earlier.  It should be noted that in  $\cH_1$  terms with $K^{\frac{j}{n+1}}$ still depend on $\ka$ when $j\le n$ whereas that statement holds in $\cH$ only for $j\le 1$.

Special care will be taken to show that each transformation is convergent, taking domain to domain and that the precise differentiability of the $p_i(t)$'s is observed. 
The change of variables is constructed by the method of Lie transforms of Deprit \cite{andre}.  
See   \cite{lietutoral,meyer} for the complete details of the Lie transformation method and for the source for our notation.

To this end we introduce a parameter  $\varepsilon$ and consider the Hamiltonian 
\begin{equation}\label{Hlowerstar} 
\begin{array}{rcl}
H_*(K,\ka,t,\varepsilon) & = & \ds \frac {n+1}{2n} K^{\frac{2n}{n+1}}+
\sum_{j=1}^{2n-1}K^{\frac{j}{n+1}}\varepsilon^{2n-j}f_j(\ka,t),
\\
&=&\ds \sum_{i=0}^{\infty}\frac {\varepsilon^i}{i!}H_i^0(K,\ka,t),
\end{array}
\end{equation}
where 
\begin{equation}\label{Hupper0}
\begin{array}{rcl}
\ds
H_0^0&=&\frac{n+1}{2n}K^{\frac{2n}{n+1}},\\ \\ \ds
H_i^0&= &i! K^{\frac{2n-i}{n+1}}f_{2n-i}(\ka,t), \quad \hbox{ for } i=1,2,\ldots,2n-1,\\ \\ \ds
H_i^0&=&0, \qquad\qquad\qquad\qquad\; \hbox{ for } i=2n,2n+1,\ldots.
\end{array}
\end{equation}
  The parameter $\varepsilon$ is usually consider  small so that it generates a near identity transformation, but in our case the original Hamiltonian $H$ in (\ref{originalKk}) is  obtained from $H_*$ in (\ref{Hlowerstar}) by setting $\varepsilon=1$.  Therefore we need to construct the change of variables, which is valid and convergent  when $\varepsilon=1$.  This is accomplished by taking only a finite number terms in the generating function $W$ and with careful estimates.

The general procedure is to  expand everything in the parameter $\varepsilon$ and use the following notation.  Introduce a double indexed 
array of functions $H_j^i$ so that the  Hamiltonian is $H_*$ in $(\ref{Hlowerstar})$ is transformed to the  Hamiltonian 
\begin{equation}\label{Hupperstar}
H^*(K,\ka,t,\varepsilon)=\sum_{i=0}^\infty \frac{\varepsilon^i}{i!}H^i_0 (K,\ka,t).
\end{equation}
The generating function for the transformation is
\begin{equation}\label{gen}
W(K,\ka,t,\varepsilon)=\sum_{k=0}^\infty \frac{\varepsilon^k}{k!} W_{k+1} (K,\ka,t).
\end{equation}
One computes the transformation via a Lie triangle, whose entries are given by 
\begin{equation}\label{lietriangle}
H_j^i=H_{j+1}^{i-1}+\sum_{k=0}^j \binom{j}{k}  \{H_{j-k}^{i-1},W_{k+1}\}.
\end{equation}
The interdependence of the functions $\{H^i_j\}$ can easily be understood by
considering the Lie triangle
$$
\begin{array}{ccccc}
H_0^0
\\
\downarrow
\\
H_1^0 & \rightarrow & H_0^1
\\
\downarrow & & \downarrow
\\
H_2^0 & \rightarrow & H_1^1 & \rightarrow & H_0^2 \, 
\\
\downarrow & & \downarrow & & \downarrow
\end{array}
$$
The coefficients of the expansion of the old function $H_*$ are in the left
column, and those of the new function $H^*$ are on the diagonal.  Formula (\ref{lietriangle})
states that to calculate any element in the Lie triangle, one needs the entries
in the column one step to the left and up.

Since the Hamiltonians depend on $t$ the remainder function must be computed by a similar Lie triangle, which  gives rise to differentiability  requirements on the $p_j(t)$'s. 
In our case the dependency on $t$ is not initially important and our goal is only to make  the first few terms $H_0^i$  independent of $\ka$.  As a benefit of this approach  we can compute the remainder function $R$ as the transformation of $ -\partial  W/\partial t$  after $W(K,\ka,t,\ep)$ has been determined.

Each transformation will be done in three steps. First the Lie transformation will be applied ignoring the $t$ dependence, 
then the transformation is shown to be convergent up to $\ep=1$, and finally the remainder term will be computed.
The three steps given below are for the first transformation and then the modifications for the second transformation will
be discussed.

\subsubsection*{The First Transformation}
We will use the algorithm summarized in Theorem 10.3.1 of \cite{meyer} and to that end we introduce three sequences of linear spaces $\mathcal{P}_r,\,\mathcal{Q}_r,\,\mathcal{R}_r$ where $r$ is a row index, $r=0,1,\ldots$.  Specifically

\begin{tabular}{ll}
$\mathcal{P}_r$ is the set of all functions of the form $K^{\frac {2n-r}{n+1}}F(\ka,t)$, & {\em (Row terms),} 
\\
$\mathcal{Q}_r$ is the set of all functions of the form $K^{\frac {2n-r}{n+1}}\overline{F}(t)$,& {\em (Reduced terms),}
\\
$\mathcal{R}_r$ is the set of all functions of the form $K^{\frac {n-r+1}{n+1}}\tilde{F}(\ka,t)$, & {\em (W terms),}
\end{tabular}

\noindent
where $F(\ka,t)$ is $4\tau$-periodic in $\ka$ and $T$-periodic in $t$, $\tilde{F}(\ka,t)$ is $4\tau$-periodic in $\ka$ with mean value zero and $T$-periodic in $t$, and $\overline{F}(t)$ is $T$-periodic in $t$.

Now check the hypotheses.  Clearly $H_r^0\in \mathcal{P}_r$ and $\mathcal{Q}_r \subset \mathcal{P}_r$. To check that  
$\{\mathcal{P}_r,\mathcal{R}_s\} \subset \mathcal{P}_{r+s}$ let $A=K^{(2n-r)/(n+1)}F_r(\ka,t)\in \mathcal{P}_r$ and 
 $B=K^{(n-s+1)/(n+1)}\tilde{F}_s(\ka,t)\in \mathcal{R}_s$. Since the functions $F_r$ and $F_s$ are generic functions it is enough to check  
the powers of $K$ in  
$$
\{A,B\} = \frac {\partial A}{\partial K}\frac {\partial B}{\partial \ka}-\frac {\partial A}{\partial \ka}\frac {\partial B}{\partial K}.
$$
The powers are
$$
\left(\frac {2n-r}{n+1}-1 \right)+\left(\frac {n-s+1}{n+1}\right)=\left(\frac {2n-r}{n+1} \right)+\left(\frac {n-s+1}{n+1}-1\right)=\frac {2n-r-s}{n+1}
$$
and therefore $\{A,B\}\in \mathcal{P}_{r+s}$.

   Next we need to show that for any $D\in \mathcal{P}_r$ there is a solution pair $B\in \mathcal{Q}_r,\,C\in \mathcal{R}_r$ that satisfy the Lie equation
$$
B=D+\{H_0^0,C\}.
$$
Given $D=K^{\frac {2n-r}{n+1}}F_r(\ka,t)$ define $B=K^{\frac {2n-r}{n+1}}\overline{F}_r(t)$ where 
$\overline{F}_r(t)$  is the $\ka$ mean value of $F_r(\ka,t)$ and seek 
$C=K^{\frac {n-r+1}{n+1}}\tF_r(\ka,t)$. We need to solve
$$
0=K^{\frac {2n-r}{n+1}}(F_r(\ka,t)-\overline{F}_r(t))+K^{\frac {2n-r}{n+1}}\frac {\partial \tF_r}{\partial \ka}(\ka,t),
$$
and 
$$
\tF_r(\ka,t)=-\int_0^{\ka}(F_r(k,t)-\overline{F}_r(t))dk
$$
does the trick and with it we have $W_r(K,\ka,t)=K^{\frac {n-r+1}{n+1}}\tF_r(\ka,t)$.

We stop computing new $W$ terms after  $n$ rows and set $W_{j}=0$ for $j\ge n$, so we have constructed a generating function
$$
W(K,\ka,t,\ep)=\sum_{j=0}^{n-1}\frac{\ep^j}{j!}W_{j+1}(K,\ka,t)=
\sum_{j=0}^{n-1}\frac {\ep^j}{j!}K^{\frac {n-j}{n+1}}\tF_j(\ka,t).
$$
Thus $W$  transforms (\ref{Hlowerstar}) to (\ref{Hupperstar}) where the terms are of the form

\begin{eqnarray}\label{H0j}
H_0^0&=&\frac{n+1}{2n}K^{\frac{2n}{n+1}},  \label{H00}
\\  
H^j_0&=& K^{\frac{2n-j}{n+1}}\overline{F}_j(t)\quad\qquad \hbox{ for } j=1,\ldots,n,\label{Hi0}
\\ 
H^j_0&=& K^{\frac{2n-j}{n+1}}F_j(\ka,t) \qquad \hbox{ for } j=n+1,\ldots, \infty. \label{Hj0}
\end{eqnarray}

So far the Lie procedure is formal, but the constructed generating function $W$ is  finite so is a convergent series.  Moreover closer inspection reveals that  the following Lemma applies.

\begin{lemma}\label{lemma1}
There exists a constant $\K>0$ such that the transformation generated by $W(K,\ka,t,\ep)$ is uniformly convergent for $K>\K$,
$0 \le \ep \le 1$, all $\ka$ and $t$.  In particular the transformation takes $H_*$ to $H^*$ when $\ep=1$.
\end{lemma}
\begin{proof}  
Look at  the $K$ equation for the transformation 
$$
 \frac {dK}{d\ep}= \ds \frac{\partial W}{\partial \ka}
=
\sum_{j=0}^{n-1} \frac {\ep^j}{j!}K^{\frac {n-j}{n+1}}\frac {\partial \tF_j}{\partial \ka}(\ka,t),
$$
$$
(n+1) \frac {dK^{\frac 1{n+1}}}{d\ep}=
\sum_{j=0}^{n-1} \frac {\ep^j}{j}K^{\frac {-j}{n+1}}\frac {\partial \tF_j}{\partial \ka}(\ka,t),
$$
or with $\Lambda=K^{\frac 1{n+1}}$
$$
(n+1)\frac {d\Lambda}{d\ep}=\sum_{j=0}^{n-1} \frac {\ep^j}{j!}\Lambda^{-j}\frac {\partial \tF_j}{\partial \ka}(\ka,t).
$$
Now let $(n+1)A=\max |\partial \tF_j/\partial \ka|$ so that
$$
-A\sum_{j=0}^{n-1}\frac {\ep^j}{j!}\Lambda^{-j} \le \frac {d \Lambda}{d\ep}\le
A\sum_{j=0}^{n-1}\frac {\ep^j}{j!}\Lambda^{-j} .
$$
Assume $K\ge 1$ and for $0\le \ep\le 1$  use $\sum_{j=0}^{n-1}\ep^j/j!<B$ so that 
$$
-AB\Lambda^{1-n} \le \frac {d \Lambda}{d\ep}\le AB.
$$
Integrating these inequalities gives
$$
\Lambda_0^nAB\ep  \le \Lambda^n  \le (\Lambda_0+AB\ep)^n.
$$
Thus if $\Lambda_0^n \ge1+ nAB$ or $K_0 \ge (1+nAB)^{\frac {n+1}n}$ the equations can be integrated all the way up to $\ep =1$ and $K(\ep)>1$.

The $\ka$ equation for the transformation is
$$
\ds \frac {d\ka}{d\ep} = \ds -\frac {\partial W}{\partial K}= \ds
-\sum_{j=0}^{n-1} \frac {\ep^j}{j!}\left(\frac {n-j}{n+1}\right)
K^{\frac {-1-j}{n+1}}\tF_j(\ka,t)
$$
Now let $C=\max |\tF_j(\ka,t)|  $  and with $\ds \sum_{j=0}^{n-1} \frac {\ep^j}{j!}\left(\frac {n-j}{n+1}\right)<B $ we have 
$$-BC<\frac {d\ka}{d\ep}<BC $$
  so that also $\ka(\ep)$ exists up until $\ep=1$.
\end{proof}

In order to account for the time dependency we must compute the remainder function $R$, which  is the  transform of $-\partial W/{\partial t}$.
 Since $W$ only involves $p_j(t)$ for $j=n,\ldots,2n-2$ they must be at least $C^1$ so far, but  $p_j(t)$ for $j=0,\ldots,n-1$ have not yet appeared in $W$ so that they only need to be continuous for the first transformation.

 More specifically we have to transform
$$
-\frac{\partial W}{\partial t}=-\sum_{j=0}^{n-2} \frac{\ep^j}{j!} \frac{\partial W_{j+1}}{\partial t}
$$
via another Lie triangle using the same generating function $W$. If the entries for that triangle are denoted by  $R_j^i$  and  
$$ 
R_j^0 = -\frac{\partial W_{j+1}}{\partial t}=-K^{\frac{n-j}{n+1}} \frac{\partial \tF_{j}}{\partial t}(\ka,t) \qquad\hbox{for }\; j=0,1,\ldots
$$
then the remainder function is 
$$
R(K,\ka,t,\ep)=\sum_{i=0}^\infty \frac {\ep^i}{i!}R_0^{i}(K,\ka,t).
$$

The construction of $R$ follows the same argument as given above with two exceptions.  First $W$ is already known with it's entries in $\mathcal{R}_j$ and second the beginning entries are $R_j^0$ are also in $\mathcal{R}_j$ not in $\mathcal{P}_j$.  Thus all the entries of the triangle, $R^i_{j-i}$ are in $\mathcal{R}_j$.  That means  that 
\begin{equation} \label{remainder} 
R_0^j = K^{\frac{n-j}{n+1}} \tilde{G}_j(\ka,t) \quad \hbox{for}\; j=0, 1,\ldots,\infty 
\end{equation}
where $\tilde{G}_j(\ka,t)$ is another periodic function.  

Thus at the end of the first transformation we arrive at the intermediate Hamiltonian 
$$ 
\mathcal{H}_1=H_0^0+\sum_{i=1}^{\infty}\frac 1{i!}(H_0^i+R_0^{i-1}).
$$
The terms for $H_0^i$ are given in (\ref{H00}) to  (\ref{Hj0}). Since $R_0^{i-1}$ contains $K^\frac{n-i}{n+1}$
 it is added to those of (\ref{Hj0}) when terms with the same powers of $K$ are combined. On the other hand combining the terms  does not change those in (\ref{Hi0}).  Thus we  arrive at the  form which was given in (\ref{intermediateKk}).

\subsubsection*{The Second Transformation}  This time we will transform  $\cH_1$ to $\cH$.  So consider 
\begin{equation}\label{intermediate}
H_*(K,\ka,t,\ep)= \sum_{i=0}^\infty \frac {\ep^i}{i!}H_i^0(K,\ka,t)
\end{equation}
where now
\begin{equation}
\begin{array}{rcl}
H_0^0&=&\frac{n+1}{2n}K^{\frac{2n}{n+1}},\\ \\
H_i^0&= & i! K^{\frac{2n-i}{n+1}}\bar{f}_i(t), \quad \hbox{ for } i=1,\ldots,n,\\ \\
H_i^0&=& i!K^{\frac{2n-i}{n+1}}f_i(\ka,t), \quad \hbox{ for } i=n+1,\ldots .
\end{array}
\end{equation}
The intermediate Hamiltonian $\mathcal{H}_1$ is obtained from (\ref{intermediate}) by setting $\ep=1$.  Since the first $n$ rows already have the desired form we set $W_r=0$ for $r=1,\ldots,n$  and determine $W_r$ for $r=n+1,\ldots,2n$ so that the  terms  $H_0^r$ do not depend on $\ka$.  We also set $W_r=0$ for $r=2n+1,\ldots$, so that Lemma \ref{lemma1} can be used which  shows that also this transformation  is convergent for $\ep=1$.  The remainder is calculated as before and with it we find 
$$ 
\mathcal{H}=H_0^0+\sum_{i=1}^{\infty}\frac 1{i!}(H_0^i+R_0^{i-1}).
$$
Finally by grouping terms with the same powers of $K$ we see that we have obtained $\cH$ in the form as displayed in (\ref{HSCKk}).  

  The first transformation removed the $\ka$ dependences for terms in rows $r=1,\ldots,n-1$ of the Lie triangle, and the remainder function required that the $p_j(t),\,j=n,\ldots,2n-2$ be $C^1$.  The second transformation removed the $\ka$ dependences for terms which appear in rows $r=n,\ldots,2n-2$. It required that the $p_j(t),\,j=1,\ldots,n-1$ be $C^1$, but  the $p_j(t)$ for $j=n,\ldots,2n-2$ also appeared in the new generating function $W$ so that they must be $C^2$ in total. However $p_0(t)$ did not occur in either of the generating functions  so that it needs only to be $C^0$.

\end{document}